\documentclass[12pt]{article}
\usepackage{amssymb}
\usepackage{mathrsfs}
\usepackage[hypertex]{hyperref}
\usepackage{amsfonts}
\usepackage{graphicx}
\usepackage{mathptmx}
\usepackage{latexsym,amsmath,amssymb,amsfonts,amsthm}

\newtheorem{theorem}{Theorem}[section]
\newtheorem{lemma}[theorem]{Lemma}

\newtheorem{remark}[theorem]{Remark}

\newtheorem{definition}[theorem]{Definition}

\begin{document}
\setcounter{page}{1}
\title{Eigenvalue comparisons in Steklov eigenvalue problem and some other eigenvalue estimates}
\author{Yan Zhao$^{1}$,~~Chuanxi Wu$^{1}$,~~Jing Mao$^{1,2,\ast}$,~~Feng Du$^{3}$}
\date{}
\protect\footnotetext{\!\!\!\!\!\!\!\!\!\!\!\!{$^{\ast}$Corresponding
author}
}
\maketitle ~~~\\[-15mm]
\begin{center}{\footnotesize
$^{1}$Faculty of Mathematics and Statistics, \\
Key Laboratory of Applied Mathematics of Hubei Province, \\
Hubei University, Wuhan 430062, China\\
$^{2}$ Department of Mathematics, Instituto Superior T\'{e}cnico, University of Lisbon,\\
Av. Rovisco Pais, 1049-001 Lisbon, Portugal\\
 Emails: jiner120@163.com, jiner120@tom.com\\
 $^{3}$School of Mathematics and Physics
Science,\\
Jingchu University of Technology, Jingmen, 448000, China}
\end{center}

%\\[5mm]
\begin{abstract}
In this paper, two interesting eigenvalue comparison theorems for
the first non-zero Steklov eigenvalue of the Laplacian have been
established for manifolds with radial sectional curvature bounded
from above. Besides, sharper bounds for the first non-zero
eigenvalue of the Wentzell eigenvalue problem of the weighted
Laplacian, which can be seen as a natural generalization of the
classical Steklov eigenvalue problem, have been obtained.

\end{abstract}

\markright{\sl\hfill  Y. Zhao, C.-X. Wu, J. Mao, F. Du\hfill}

\section{Introduction}
\renewcommand{\thesection}{\arabic{section}}
\renewcommand{\theequation}{\thesection.\arabic{equation}}
\setcounter{equation}{0} \setcounter{maintheorem}{0}

Let $(M,g)$ be an $n$-dimensional $(n\geq2)$ complete Riemannian
manifold with the Riemannian metric $g$, and let $\Omega\subseteq M$
be a compact domain with boundary $\partial\Omega$. The so-called
\emph{Steklov eigenvalue problem} is actually to find a solution of
the following system
\begin{eqnarray} \label{SEP}
\left\{
\begin{array}{ll}
\Delta\varphi=0  \qquad\quad & {\rm{in}}~\Omega,\\
\frac{\partial\varphi}{\partial\vec{\eta}}=v\varphi \qquad\quad &
{\rm{on}}~\partial\Omega,
\end{array}
\right.
\end{eqnarray}
where $\Delta$ is the Laplacian on $M$ associated with the metric
$g$, $\vec{\eta}$ is the unit outward normal vector field of the
boundary $\partial\Omega$, and $v$ is a real number called the
\emph{eigenvalue} of this problem. There are infinitely many real
numbers $v$ satisfying the system (\ref{SEP}) and can be listed
increasingly as a sequence tending to the infinity. Clearly, the
first Steklov eigenvalue of the problem (\ref{SEP}) is zero with the
constant function as its eigenfunction. By the variational
principle, it is not difficult to get the first non-zero Steklov
eigenvalue $v_{1}(\Omega)$ is characterized by
\begin{eqnarray} \label{cha}
v_{1}(\Omega)=\min\limits_{\int_{\partial\Omega}u=0}\frac{\int_{\Omega}|\nabla
u|^2}{\int_{\partial\Omega}u^2},
\end{eqnarray}
where $\nabla$ is the gradient operator on $\Omega\subseteq M$, and
$u\in W^{1,2}(\Omega)$, the completion of the set of smooth
functions under the Sobolev norm $\|
u\|_{1,2}=\int_{\Omega}u^{2}+\int_{\Omega}|\nabla u|^{2}$.
 The problem (\ref{SEP}) was introduced by
Steklov \cite{mws} with the physical background as follows: the
function $\varphi$ denotes the steady state temperature on $\Omega$
such that the flux on $\partial\Omega$ is proportional to the
temperature. Since the set of eigenvalues for the Steklov eigenvalue
problem is the same as the set of eigenvalues of the well-known
Dirichlet-to-Neumann map, the problem (\ref{SEP}) has important
influence in the study of conductivity and harmonic analysis, which
was initially studied by Calder\'{o}n \cite{apc}. Anyway, Escobar
\cite{e1} showed that the study of (\ref{SEP}) is also important in
the problem of conformal deformation of a Riemannian metric on
manifolds with boundary.

By (\ref{cha}), it is easy to get the following Sobolev trace
inequality
\begin{eqnarray*}
\int_{\partial\Omega}|u-u_{0}|^{2}\leq\frac{1}{v_{1}(\Omega)}\int_{\Omega}|\nabla
u|^{2},
\end{eqnarray*}
where $u_{0}$ is the mean  value of the function $u$ when restricted
to the boundary. This inequality makes an important role in the
study of existence and regularity of solutions of some boundary
value problems.

In order to state our main conclusions below clearly,  here we would
like to introduce some basic notions, which have been introduced in
\cite{fmi,jm1,jm2,mdw} already. Besides, in the sequel, for
convenience, \emph{we will drop the integral measures for all
integrals except it is necessary}.

\subsection{Basic notions}

As before, let $(M,g)$ be an $n$-dimensional ($n\geq2$) complete
Riemannian manifold with the metric $g$, and $\nabla$ be the
gradient operator. For a point $p\in M$, one can set up a geodesic
polar coordinates $(t,\xi)$ around this point $p$, where
$\xi\in{S}_{p}^{n-1}\subseteq{T_{p}M}$ is a unit vector of the unit
sphere $S_{p}^{n-1}$ with center $p$ in the tangent space $T_{p}M$.
Let $\mathcal{D}_{p}$, a star shaped set of $T_{p}M$, and $d_{\xi}$
be defined by
\begin{eqnarray*}
\mathcal{D}_{p}=\{t\xi|~0\leq{t}<d_{\xi},~\xi\in{S^{n-1}_{p}}\},
\end{eqnarray*}
and
\begin{eqnarray*}
d_{\xi}=d_{\xi}(p):=  \sup\{t>0|~\gamma_{\xi}(s):= \exp_p(s\xi)~
{\rm{is~ the~ unique~ minimal~ geodesic ~joining} }~ p ~{\rm{and}}
~\gamma_{\xi}(t)\}
\end{eqnarray*}
respectively. Then $\exp_p:\mathcal{D}_p \to M\backslash Cut(p)$
gives a diffeomorphism from $\mathcal{D}_p$ onto the open set
$M\backslash Cut(p)$, with $Cut(p)$ the cut locus of  $p$. For
$\zeta\in{\xi^{\bot}}$, one can define the path of linear
transformations $\mathbb{A}(t,\xi):\xi^\perp\rightarrow{\xi^\perp}$
by
 \begin{eqnarray*}
\mathbb{A}(t,\xi)\zeta=(\tau_{t})^{-1}Y(t),
 \end{eqnarray*}
with $\xi^\perp$ the orthogonal complement of $\{\mathbb{R}\xi\}$ in
$T_{p}M$, where $\tau_{t}:T_{p}M\rightarrow{T_{\exp_{p}(t\xi)}M}$ is
the parallel translation along the geodesic $\gamma_{\xi}(t)$ with
$\gamma'(0)=\xi$, and $Y(t)$ is the Jacobi field along
$\gamma_{\xi}$ satisfying $Y(0)=0$, $(\triangledown_{t}Y)(0)=\zeta$.
Moreover, set
\begin{eqnarray*}
\mathcal{R}(t)\zeta=(\tau_{t})^{-1}R(\gamma'_{\xi}(t),\tau_{t}\zeta)
\gamma'_{\xi}(t),
\end{eqnarray*}
where the curvature tensor $R(X,Y)Z$ is defined by
$R(X,Y)Z=-[\nabla_{X},$ $ \nabla_{Y}]Z+ \nabla_{[X,Y]}Z$. Then
$\mathcal{R}(t)$ is a self-adjoint operator on $\xi^{\bot}$, whose
trace is the radial Ricci tensor
$$\mathrm{Ric}_{\gamma_{\xi}(t)}(\gamma'_{\xi}(t),\gamma'_{\xi}(t)).$$
Clearly, the map $\mathbb{A}(t,\xi)$ satisfies the Jacobi equation
 $\mathbb{A}''+\mathcal{R}\mathbb{A}=0$ with initial conditions
 $\mathbb{A}(0,\xi)=0$, $\mathbb{A}'(0,\xi)=I$, and  by Gauss's lemma, the Riemannian metric of $M$ can be
 expressed by
 \begin{eqnarray} \label{2.1}
 ds^{2}(\exp_{p}(t\xi))=dt^{2}+|\mathbb{A}(t,\xi)d\xi|^{2}
 \end{eqnarray}
on the set $\exp_{p}(\mathcal{D}_p)$. Consider the metric components
$g_{ij}(t,\xi)$, $i,j\geq 1$, in a coordinate system $\{t, \xi_a\}$
formed by fixing  an orthonormal basis $\{\zeta_a, a\geq 2\}$ of
 $\xi^{\bot}=T_{\xi}S^{n-1}_p$, and extending it to a local frame $\{\xi_a, a\geq2\}$ of
$S_p^{n-1}$. Define
 a function $J>0$ on $\mathcal{D}_{p}$ by
\begin{equation*}
J^{n-1}=\sqrt{|g|}:=\sqrt{\det[g_{ij}]}.
\end{equation*}
 Since
 $\tau_t: S_p^{n-1}\to S_{\gamma_{\xi}(t)}^{n-1}$ is an
isometry, we have
$$
g\left( d(\exp_p)_{t\xi}(t\zeta_{a}),
d(\exp_p)_{t\xi}(t\zeta_{b})\right)=g\left(
\mathbb{A}(t,\xi)(\zeta_{a}), \mathbb{A}(t,\xi)(\zeta_{b})\right),
$$
and
$$ \sqrt{|g|}=\det\mathbb{A}(t,\xi).$$
So, by (\ref{2.1}), the volume $V(B(p,r))$ of the geodesic ball
$B(p,r)$ on $M$ is given by
\begin{eqnarray} \label{2.2}
V(B(p,r))=\int_{S_{p}^{n-1}}\int_{0}^{\min\{r,d_{\xi}\}}\sqrt{|g|}dtd\sigma=\int_{S_{p}^{n-1}}\left(\int_{0}^{\min\{r,d_{\xi}\}}\det(\mathbb{A}(t,\xi))dt\right)d\sigma,
\end{eqnarray}
where $d\sigma$ denotes the $(n-1)$-dimensional volume element on
$\mathbb{S}^{n-1}\equiv S_{p}^{n-1}\subseteq{T_{p}M}$. The
injectivity radius $\mathrm{inj}(p)$ at $p$ satisfies
 \begin{eqnarray*}
\mathrm{inj}(p):=d(p,Cut(p))=\min_{\xi}d_{\xi}.
 \end{eqnarray*}
 In general, we have
$B(p,\mathrm{inj}(p))\subseteq{M}\backslash{Cut(p)}$. Besides, for
$r<\mathrm{inj}(p)$, by (\ref{2.2}) we can obtain
\begin{eqnarray*}
V(B(p,r))=\int_{0}^{r}\int_{S_{p}^{n-1}}\det(\mathbb{A}(t,\xi))d\sigma{dt}.
\end{eqnarray*}
Denote by $r(x)=d(x,p)$ the intrinsic distance to the point
$p\in{M}$. Then, by the definition of a non-zero tangent vector
``\emph{radial}" to a prescribed point on a manifold given in the
first page of \cite{KK}, we know that for
$x\in{M}\backslash(Cut(p)\cup{p})$ the unit vector field
\begin{eqnarray*}
v_{x}:=\nabla{r(x)}
\end{eqnarray*}
is the radial unit tangent vector at $x$. This is because for any
$\xi\in{S}_{p}^{n-1}$ and $t_{0}>0$, we have
$\nabla{r}{(\gamma_{\xi}(t_{0}))}=\gamma'_{\xi}(t_{0})$ when the
point $\gamma_{\xi}(t_{0})=\exp_{p}(t_{0}\xi)$ is away from the cut
locus of $p$.

We need the following concepts.

\begin{definition} (\cite{fmi,jm1,jm2}) \label{def1}
 Given a continuous function
$k:[0,l)\rightarrow \mathbb{R}$, we say that $M$ has  a radial Ricci
curvature lower bound $(n-1)k$ along any unit-speed minimizing
geodesic starting from a point $p\in{M}$ if
\begin{eqnarray}  \label{2.4}
{\rm{Ric}}(v_x,v_x)\geq(n-1)k(r(x)), ~~\forall x\in
M\backslash\left(Cut(p)\cup\{p\}\right),
\end{eqnarray}
where ${\rm{Ric}}$ is the Ricci curvature of $M$.
\end{definition}

\begin{definition}(\cite{fmi,jm1,jm2})  \label{def2}
Given a continuous function $k:[0,l)\rightarrow \mathbb{R}$, we say
that $M$ has a radial sectional curvature upper bound $k$ along any
unit-speed minimizing geodesic starting from a point $p\in{M}$ if
\begin{eqnarray} \label{2.5}
\mathcal{K}(v_{x},V)\leq{k(r(x))},  ~~\forall x\in
M\backslash\left(Cut(p)\cup\{p\}\right),
\end{eqnarray}
where $V\perp{v_{x}}$, $V\in{S^{n-1}_{x}}\subseteq{T_{x}M}$, and
$\mathcal{K}(v_{x},V)$ is the sectional curvature of the plane
spanned by $v_x$ and $V$.
\end{definition}

\begin{remark} \label{remark2-3}
\rm{Since $r(x)=d(p,x)=t$ and $\frac{d}{dt}|_{x}=\nabla{r(x)}=v_x$,
we know that the inequalities (\ref{2.4}) and (\ref{2.5}) become
${\rm{Ric}}(\frac{d}{dt},\frac{d}{d{t}}) \geq(n-1)k(t)$ and
$\mathcal{K}(\frac{d}{dt},V)\leq{k(t)}$, respectively. This fact has
been pointed out in \cite[Remark 2.4]{fmi} or \cite[Remark
2.1.5]{jm1}. Besides, for convenience, if a manifold satisfies
(\ref{2.4}) (resp., (\ref{2.5})), then we say that $M$ has \emph{a
radial Ricci curvature lower bound w.r.t. a point $p$} (resp.,
\emph{a radial sectional curvature upper bound w.r.t. a point $p$}),
that is to say, its \emph{radial Ricci curvature is bounded from
below w.r.t. $p$} (resp., \emph{radial sectional curvature is
bounded from above w.r.t. $p$}). As also pointed out in \cite[page
706]{fmi} or \cite[page 378]{jm2}, for a given complete Riemannian
$n$-manifold $(M,g)$, one can define
\begin{eqnarray*}
k_{-}(p,t):=\min\limits_{\{\xi|(t,\xi)\in\overline{\mathbb{D}}_{p}\}}
\frac{{\rm{Ric}}_{\gamma_{\xi}(t)}\left(\frac{d}{dt}|_{\exp_{p}(t\xi)},\frac{d}{dt}|_{\exp_{p}(t\xi)}\right)}{n-1},
\qquad 0\leq t<l(p):=\max_{\xi}d_{\xi},
\end{eqnarray*}
and
 \begin{eqnarray*}
k_{+}(p,t):=\max_{\{(\xi,V)|
\gamma'_{\xi}(t)\perp{V},|V|=1\}}\mathcal{K}_{\gamma_{\xi}(t)}\left(\frac{d}{dt}\Big{|}_{\exp_{p}(t\xi)},V\right),
\qquad 0\leq t<\mathrm{inj}(p),
\end{eqnarray*}
where
$\mathbb{D}_{p}:=\{(t,\xi)\in[0,\infty)\times{S_{p}^{n-1}}|0\leq
t<d_{\xi}\}$ with closure
$\overline{\mathbb{D}}_{p}=\{(t,\xi)\in[0,\infty)\times{S_{p}^{n-1}}|0\leq
t\leq d_{\xi}\}$, and, by applying the uniform continuity of
continuous functions on compact sets, for a bounded domain
$\Omega\subseteq{M}$, one can always find \emph{optimally
continuous} bounds $k_{\pm}(p,t)$ for the radial sectional and Ricci
curvatures w.r.t. some point $p\in\Omega$. }
\end{remark}

We need the following notion of spherically symmetric manifolds.

\begin{definition} (\cite{fmi,jm1,jm2}) \label{defsp}
A domain
$\Omega=\exp_{p}([0,l)\times{S_{p}^{n-1}})\subset{M}\backslash
Cut(p)$, with $l<\mathrm{inj}(p)$, is said to be spherically
symmetric with respect to a point $p\in\Omega$, if the matrix
$\mathbb{A}(t,\xi)$ satisfies
$\mathbb{A}(t,\xi)=f(t)\mathrm{I}_{(n-1)\times(n-1)}$, for a
function $f\in{C^{2}([0,l])}$, $l\in(0,\infty]$ with $f(0)=0$,
$f'(0)=1$, $f|(0,l)>0$, where $\mathrm{I}_{(n-1)\times(n-1)}$
represents the $(n-1)\times(n-1)$ identity matrix.
\end{definition}

By (\ref{2.1}), on the set
 $\Omega$ given in Definition \ref{defsp} the Riemannian metric of $M$ can be
 expressed by
 \begin{eqnarray} \label{2.3}
 ds^{2}(\exp_{p}(t\xi))=dt^{2}+f^{2}(t)|d\xi|^{2}, \qquad
 \xi\in{S_{p}^{n-1}}, \quad 0\leq{t}<l,
 \end{eqnarray}
 with $|d\xi|^{2}$ the round metric on the unit sphere $\mathbb{S}^{n-1}\subseteq\mathbb{R}^{n}$.
 Spherically symmetric manifolds were named as generalized space
forms
 by Katz and Kondo \cite{KK}, and a standard model for
such manifolds is given by the quotient manifold of the warped
product $[0,l)\times_{f} \mathbb{S}^{n-1}$ equipped with the metric
(\ref{2.3}),  and all pairs $(0,\xi)$ are identified with a single
point $p$, where $f$ satisfies the conditions in Definition
\ref{defsp}, and is called \emph{the warping function}.  A space
form with constant curvature $k$ is also a spherically symmetric
manifold, and in this special case we have
\begin{eqnarray*}
f(t)=\left\{
\begin{array}{llll}
\frac{\sin\sqrt{k}t}{\sqrt{k}}, & \quad  l= \frac{\pi}{\sqrt{k}}
  & \quad k>0,\\
 t, &\quad l=+\infty & \quad k=0, \\
\frac{\sinh\sqrt{-k}t}{\sqrt{-k}}, & \quad l=+\infty  &\quad k<0.
\end{array}
\right.
\end{eqnarray*}

\subsection{The statement of main conclusions} \label{sub1.2}

In \cite{fmi}, Freitas, Mao and Salavessa proved that for
$n$-dimensional ($n\geq2$) complete Riemannian manifold $(M,g)$
having a \emph{radial Ricci curvature lower bound $(n-1)k(t)$ w.r.t.
a point $p\in M$}, the first Dirichlet eigenvalue
$\lambda_{1}(B(p,r))$ of the Laplacian on $B(p,r)$ satisfies
 \begin{eqnarray} \label{e1}
\lambda_{1}(B(p,r))\leq \lambda_{1}(\mathcal{B}_{n}(p^{-},r)),
 \end{eqnarray}
where $\lambda_{1}(\mathcal{B}_{n}(p^{-},r))$ is the first Dirichlet
eigenvalue of the Laplacian on the geodesic ball
$\mathcal{B}_{n}(p^{-},r)$, with the center $p^{-}$ and radius $r$,
of the spherically symmetric $n$-manifold
$M^{-}=[0,l)\times_{f}\mathbb{S}^{n-1}$ with the base point $p^{-}$
and the warping function $f$ determined by
\begin{eqnarray} \label{ODE}
\left\{
\begin{array}{ll}
f''(t)+k(t)f(t)=0  \qquad\quad & {\rm{on}}~(0,l),\\
f(0)=0,~f'(0)=1.
\end{array}
\right.
\end{eqnarray}
Equality in (\ref{e1}) holds if and only if $B(p,r)$ is isometric to
$\mathcal{B}_{n}(p^{-},r)$. They also proved that if $M$ has a
\emph{radial sectional curvature upper bound $(n-1)k(t)$ w.r.t. a
point $p\in M$}, then for $r<\mathrm{inj}(p)$,
 \begin{eqnarray} \label{e2}
\lambda_{1}(B(p,r))\geq \lambda_{1}(\mathcal{B}_{n}(p^{+},r))
 \end{eqnarray}
holds, where $\lambda_{1}(\mathcal{B}_{n}(p^{+},r))$ is the first
Dirichlet eigenvalue of the Laplacian on the geodesic ball
$\mathcal{B}_{n}(p^{+},r)$, with the center $p^{+}$ and radius $r$,
of the spherically symmetric $n$-manifold
$M^{+}=[0,l)\times_{f}\mathbb{S}^{n-1}$ with the base point $p^{+}$
and $f$ determined by (\ref{ODE}). Equality in (\ref{e2}) holds if
and only if $B(p,r)$ is isometric to $\mathcal{B}_{n}(p^{+},r)$.
Clearly, the model spaces, i.e., spherically symmetric manifolds,
are determined by the curvature bounds. These eigenvalue estimates
improve the classical Cheng's eigenvalue comparison theorems
\cite{cheng1,cheng2}, whose model spaces are space forms of constant
curvature, a lot in the spectral geometry. Besides, Freitas, Mao and
Salavessa \cite[Section 6]{fmi} have used interesting examples
(i.e., torus, elliptic paraboloid, saddle) to intuitively and
numerically show that their comparisons are better than those of
Cheng's. The eigenvalue comparison (\ref{e1}) has been already
extended to the case of nonlinear $p$-Laplacian ($1<p<\infty$) - see
Mao \cite[Theorem 3.2]{jm2} for the detail.

It is interesting to know whether the system (\ref{ODE}) has \emph{a
long-time existence solution} (i.e., a positive solution on
$(0,\infty)$ and in this case $l=\infty$) or not. This has close
relationship with the oscillation theory of ordinary differential
equations (ODEs for short). Clearly, by Sturm-Picone comparison
theorem, one can easily get that $l=\infty$ if $k(t)\leq0$, while
$l<\infty$ if $k(t)\geq\alpha$ for some positive constant
$\alpha>0$. Readers can check \cite[Sect. 2.6, Chapt. 2]{jm1} for
the details about the general restrictions on $k(t)$ to get the
existence or non-existence of positive solution to (\ref{ODE}) on
$(0,\infty)$.

By using spherically symmetric manifolds as the model spaces also,
we can get the following eigenvalue comparisons for the first
non-zero Steklov eigenvalue for the system (\ref{SEP}) with $\Omega$
chosen as geodesic balls.

\begin{theorem}  \label{maintheorem-1}
For a given $n$-dimensional ($n=2,3$) complete Riemannian manifold
$(M,g)$ having a radial sectional curvature upper bound $k(t)$
w.r.t. $p$, where, as before, $t:=d(p,\cdot)$ represents the
distance to the point $p\in M$, and $k(t)$ is a continuous function
w.r.t. $t$, we have
\begin{eqnarray*}
v_{1}(B(p,r))\leq v_{1}(\mathcal{B}_{n}(p^{+},r)),
\end{eqnarray*}
where $r<\mathrm{inj}(p)$, and $\mathcal{B}_{n}(p^{+},r)$ is the
geodesic ball, with the center $p^{+}$ and radius $r$, of the
spherically symmetric $n$-manifold
$M^{+}=[0,l)\times_{f}\mathbb{S}^{n-1}$ with the base point $p^{+}$
and the warping function $f$ determined by (\ref{ODE}). Equality
holds if and only if $B(p,r)$ is isometric to
$\mathcal{B}_{n}(p^{+},r)$.
\end{theorem}

For the higher dimensional case, we have the following.

\begin{theorem} \label{maintheorem-2}
Assume that same notations have the same meaning as in Theorem
\ref{maintheorem-1}. For a given $n$-dimensional ($n\geq2$) complete
Riemannian manifold $(M,g)$ having a radial sectional curvature
upper bound $k(t)$ w.r.t. $p$, if the first non-zero closed
eigenvalues of the Laplacian on the boundary satisfy
\begin{eqnarray} \label{b1}
\lambda_{1}^{c}(\partial B(p,r))\leq
\lambda_{1}^{c}(\partial\mathcal{B}_{n}(p^{+},r))
\end{eqnarray}
with $r<\mathrm{inj}(p)$, then we have
\begin{eqnarray} \label{e3}
v_{1}(B(p,r))\leq v_{1}(\mathcal{B}_{n}(p^{+},r)).
\end{eqnarray}
 Equality
holds if and only if $B(p,r)$ is isometric to
$\mathcal{B}_{n}(p^{+},r)$.
\end{theorem}

\begin{remark}
\rm{ (1) For the purse of letting readers understand the assumption
(\ref{b1}) clearly, here we would like to give a brief introduction
to the closed eigenvalue problem of the Laplacian. For an open,
bounded, connected domain $D\subset M$, without boundary, on a given
Riemannian $n$-manifold $M$, the so-called \emph{closed eigenvalue
problem} of the Laplacian on $D$ is actually to find a nontrivial
solution to $\Delta u+\lambda^{c}u=0$ in $D$, $u\in W^{1,2}(D)$.  It
is well-known that $-\Delta$ only has discrete spectrum and all the
elements in the spectrum (i.e., eigenvalues) can be listed
non-decreasingly into a sequence tending to infinity, i.e.,
$0=\lambda_{0}^{c}(D)<\lambda_{1}^{c}(D)\leq\lambda_{2}^{c}(D)\leq\lambda_{3}^{c}(D)\leq\cdots\uparrow\infty$.
For each eigenvalue $\lambda^{c}$, the corresponding $u$ to the
equation $\Delta u+\lambda^{c}u=0$ is called the eigenfunction of
$\lambda^{c}$. Clearly, the eigenfunction of $\lambda_{0}^{c}(D)=0$
is a constant function over $D$. Besides, by the variational
principle, one knows that the first non-zero closed eigenvalue
$\lambda_{1}^{c}(D)$ can be characterized as follows
\begin{eqnarray*}
\lambda_{1}^{c}(D)=\inf\left\{\frac{\int_{D}|\nabla
u|^{2}}{\int_{D}u^{2}}\Big{|}u\in
 W^{1,2}(D),u\neq0,\int_{D}u=0\right\}.
\end{eqnarray*}
 (2) The assumption (\ref{b1}) is always true for lower dimensions $n=2,3$ (see pages 16-17 for details), so Theorem \ref{maintheorem-1} looks
 much more
natural and better than Theorem \ref{maintheorem-2}.\\
(3) As pointed out by Escobar \cite[page 145]{e3}, the restraint on
the injectivity radius is necessary. In fact, Escobar constructed a
geodesic ball $B(p,r)$, with $r>\mathrm{inj}(p)$, in a flat
two-dimensional torus such that $v_{1}(B(p,r))>\frac{1}{r}$, which
is strictly bigger than the first non-zero Steklov eigenvalue
$\frac{1}{r}$ of the Euclidean $2$-ball with radius $r$. \\
 (4) Escobar \cite[pages 109-111]{e2} used non-simply-connected annuli as
the example to explain the fundamental difference between the
Dirichlet eigenvalue problem and the Steklov eigenvalue problem of
the Laplacian, which implies that the research experience on the
Dirichlet eigenvalue problem might not be used in the study of the
Steklov eigenvalue problem \emph{directly}.  \\
 (5) Clearly, one can get Escobar's main conclusions \cite[Theorems 1 and
 2]{e3} by directly applying our Theorems \ref{maintheorem-1} and
 \ref{maintheorem-2} twice. Therefore, our conclusions here are
 sharper than Escobar's and cover them as special cases. Inspired by examples constructed in \cite{fmi,jm1,jm2},
 we would
 like to use the torus example below to
 let readers realize the advantage of our comparisons here \emph{intuitively}.

\begin{itemize}

\item Let $\{x,y,z\}$ be the Cartesian coordinates of the Euclidean $3$-space $\mathbb{R}^3$. Consider the ring
torus $\mathscr{T}$ given by
\begin{eqnarray*}
\left\{
\begin{array}{lll}
x=(1+\epsilon\cdot \cos v)\cos u,\\
y=\epsilon\cdot \sin v, \\
z=(1+\epsilon\cdot \cos v)\sin u, &\quad
\end{array}
\right.
\end{eqnarray*}
with $u,v\in[0,2\pi)$ and $0<\epsilon<1$. Clearly, $\mathscr{T}$ can
be obtained by rotating the circle $(x-1)^2+y^2=\epsilon^2$ with
respect to the $y$-axis. Denote this circle by $\mathscr{C}$. It is
not difficult to know that the Gaussian curvature of $\mathcal{T}$
is given by
 \begin{eqnarray*}
 K=\frac{\cos v}{\epsilon(1+\epsilon\cdot \cos v)}.
 \end{eqnarray*}
Without loss of generality, we can choose $\epsilon=\frac{1}{2}$,
and then $K=\frac{4\cos v}{2+\cos v}$, $\mathscr{T}$ is explicitly
expressed as
\begin{eqnarray*}
\left(1-\sqrt{x^2+z^2}\right)^2+y^2=\frac{1}{4}.
\end{eqnarray*}
 As shown clearly in \cite[Example 2.5.1, Sect. 2.5,
Chapt. 2]{jm1} (or using the method explained in \cite[Example
4.4]{jm2} for finding sharp lower bounds for the Gaussian
curvature), \footnote{In the $2$-dimensional case, the radial Ricci
curvature and the radial sectional curvature degenerate into the
Gaussian curvature. Hence, here if one wants to use our conclusion
Theorem \ref{maintheorem-1} for a given parameterized surface, the
only thing needed is finding upper bounds for the Gaussian curvature
- see Remark \ref{remark2-3} for the principle of getting optimal
bounds. By Theorem \ref{maintheorem-1}, we know that the more
sharper upper bounds found, the more shaper upper bounds can be
given for $v_{1}(B(p,r))$.} for the geodesic ball $B(p,r)$ on
$\mathscr{T}$, where $0\leq r<\frac{\pi}{2}$, \footnote{This range
of $r$ is used to make sure that $B(p,r)$ is within the cut-locus of
$p$, and then Theorem \ref{maintheorem-1} can be applied directly.
For the detailed reason, please check \cite[Example 2.5.1, Sect.
2.5, Chapt. 2]{jm1} or \cite[Example 4.4]{jm2}.} one can get the
following facts:
  \\Case 1. If $p$ is one of those
points which are farthest from the $y$-axis, that is, $p$ locates on
the circle $C_{1}$ in the $xoz$-plane defined by $x^2+z^2=9/4$.
Without loss of generality, we can choose $p$ to be the point
$(3/2,0,0)$, which implies that $p$ is also on the circle
$\mathscr{C}$. Clearly, the parameter $v$ satisfies $v=0$ at the
point $p$. In this case, the best upper bound for the Gaussian
curvature what we can choose is $K_{upper}^{1}=4/3$, which implies
the upper bound for $v_{1}(B(p,r))$ given by Theorem
\ref{maintheorem-1} is the same with the one determined by Escobar's
eigenvalue comparison \cite[Theorem 1]{e3}.
 \\ Case 2. If $p$ is one of those points which are nearest to the
$y$-axis, that is, $p$ locates on the circle $C_{2}$ in the
$xoz$-plane defined by $x^2+z^2=1/4$. Without loss of generality, we
can choose $p$ to be the point $(1/2,0,0)$, which implies
$p\in\mathscr{C}$. In this case, at $p$ , $v=\pi$ and the best upper
bound for the Gaussian curvature what we can choose is
$K_{upper}^{2}=\frac{4\cos(\pi-2t)}{2+\cos(\pi-2t)}$, $0\leq
t<\frac{\pi}{2}$, where $t:=d(p,\cdot)$ is the Riemannian distance
to $p$. Moreover, the model manifold is
$M^{+}_{2}:=[0,l)\times_{f_{2}(t)}\mathbb{S}^1$ with the base point
$p^{+}_{2}$, where $f_{2}(t)$ is the solution to the system
(\ref{ODE}) with $k(t)=K_{upper}^{2}$. By Theorem
\ref{maintheorem-1} and \textbf{Fact 2} in Section 2, we have
\begin{eqnarray} \label{czmd-1}
v_{1}(B(p,r))\leq
v_{1}(\mathcal{B}_{2}(p^{+}_{2},r))=\frac{1}{f_{2}(r)},
\qquad{\mathrm{for~any}~0<r<\frac{\pi}{2}}.
\end{eqnarray}
 However, if one wants to use Escobar's conclusion
\cite[Theorem 1]{e3}, the best constant upper bound of the Gaussian
curvature should be $\widetilde{K}_{upper}^{2}=\frac{4\cos
2r}{-2+\cos2r}:=k_{0}$ since the function
$\frac{4\cos(\pi-2t)}{2+\cos(\pi-2t)}$ is increasing on $[0,r)$,
and, in this setting, the model manifold is $\mathbb{S}^{2}(k_{0})$
if $k_{0}>0$, i.e., $\frac{\pi}{4}<r<\frac{\pi}{2}$; $\mathbb{R}^2$
if $k_{0}=0$, i.e., $k_{0}=\frac{\pi}{4}$; $\mathbb{H}^{2}(k_{0})$
if $k_{0}<0$, i.e., $0\leq r<\frac{\pi}{4}$, which is actually the
$2$-dimensional space form with constant curvature $k_{0}$. Since
spherically symmetric manifolds are natural generalization of space
forms, which leads to the fact that Escobar's model space here is
actually the special spherically symmetric surface
$\widetilde{M}^{+}_{2}:=[0,\widetilde{l})\times_{\widetilde{f}_{2}(t)}\mathbb{S}^1$
(endowed with a one-point compactification topology if $k_{0}>0$),
where
\begin{eqnarray*}
\widetilde{f}_{2}(t)=\left\{
\begin{array}{llll}
\frac{\sin\sqrt{k_0}t}{\sqrt{k_0}}, & \quad  \widetilde{l}=
\frac{\pi}{\sqrt{k_0}},
  & \quad k_0>0\\
 t, &\quad \widetilde{l}=+\infty, & \quad k_0=0 \\
\frac{\sinh\sqrt{-k_0}t}{\sqrt{-k_0}}, & \quad
\widetilde{l}=+\infty, &\quad k_0<0
\end{array}
\right.
\end{eqnarray*}
is a solution to the system (\ref{ODE}) with
$k(t)=k_0=\widetilde{K}_{upper}^{2}$. Hence, by using Escobar's
conclusion \cite[Theorem 1]{e3} and \textbf{Fact 2}, one has
\begin{eqnarray} \label{ESC-1}
v_{1}(B(p,r))\leq\frac{1}{\widetilde{f}_{2}(r)},\qquad{\mathrm{for~any}~0<r<\frac{\pi}{2}}.
\end{eqnarray}
Since $f_{2}(t)$, $\widetilde{f}_{2}(t)$ satisfy the system
(\ref{ODE}) with different curvature upper bounds $K_{upper}^{2}$
and $\widetilde{K}_{upper}^{2}$, by using the fact
\begin{eqnarray*}
K_{upper}^{2}<\widetilde{K}_{upper}^{2}=k_{0}, \qquad
0<t<r<\frac{\pi}{2}
\end{eqnarray*}
and the Sturm-Picone comparison theorem, we have
$f_{2}(r)>\widetilde{f}_{2}(r)$ for any $0<r<\frac{\pi}{2}$, i.e.,
\begin{eqnarray*}
\frac{1}{f_{2}(r)}<\frac{1}{\widetilde{f}_{2}(r)},\qquad{\mathrm{for~any}~0<r<\frac{\pi}{2}},
\end{eqnarray*}
which implies, in Case 2, our upper bound for $v_{1}(B(p,r))$ given
by (\ref{czmd-1}) is sharper than the one in (\ref{ESC-1})
determined by Escobar's eigenvalue comparison \cite[Theorem 1]{e3}.
 \\ Case 3. If $p$ is neither a point on the circle $C_{1}$ nor a
point on the circle $C_{2}$. Without loss of generality, we can
choose $p$ to be a point, which is different from the points
$(3/2,0,0)$ and $(1/2,0,0)$, on the circle $\mathscr{C}$. Assume
$v=\alpha$ at $p$ with $0<\alpha<\pi$ or $\pi<\alpha<2\pi$. By the
symmetry of $\mathscr{T}$ w.r.t. the $xoy$-plane, without loss of
the generality, we can assume $0<\alpha<\pi$. In this case, the best
upper bound for the Gaussian curvature what we can choose is
\begin{eqnarray*}
K_{upper}^{3}=\left\{
 \begin{array}{lll}
\frac{4\cos(\alpha-2t)}{2+\cos(\alpha-2t)}, \qquad
0\leq{t}\leq\frac{\alpha}{2},\\
\\
\frac{4}{3}, \qquad\qquad\quad\quad
\frac{\alpha}{2}<t<\frac{\pi}{2}.
 \end{array}
 \right.
\end{eqnarray*}
If one wants to use Escobar's conclusion \cite[Theorem 1]{e3}, the
best constant upper bound of the Gaussian curvature should be
\begin{eqnarray*}
\widetilde{K}_{upper}^{3}=\left\{
 \begin{array}{lll}
\frac{4\cos(\alpha-2r)}{2+\cos(\alpha-2r)}, \qquad
0\leq{t}<r\leq\frac{\alpha}{2},\\
\\
\frac{4}{3}, \qquad\qquad\quad\quad
\frac{\alpha}{2}<t<r<\frac{\pi}{2}.
 \end{array}
 \right.
\end{eqnarray*}
Clearly, $K_{upper}^{3}<\widetilde{K}_{upper}^{3}$ for
$0\leq{t}<r\leq\frac{\alpha}{2}$, and
$K_{upper}^{3}=\widetilde{K}_{upper}^{3}$ for
$\frac{\alpha}{2}<t<r<\frac{\pi}{2}$. Using a similar argument to
Case 2, we know that for $0\leq r<\frac{\pi}{2}$, the upper bound
for $v_{1}(B(p,r))$ given by Theorem \ref{maintheorem-1} is sharper
than the one determined by Escobar's eigenvalue comparison
\cite[Theorem 1]{e3}.

\end{itemize}
(6) By \cite[Lemma 2.1]{hm}, we know that for a given complete
surface $\Sigma$, which can be parameterized, and a geodesic ball
$B(p,r)$, with $r<\mathrm{inj}(p)$, on $\Sigma$, the optimal upper
bound for the Gaussian curvature is actually the maximal values of
the Gaussian curvature on geodesic circles $C(p,t)$ with center $p$
and radii $0<t<r$, and this optimal upper bound can be computed
\emph{numerically}. This fact tells us that one can use any
parameterized complete surface, not only the ring torus example
mentioned above, to show the advantage of Theorem
\ref{maintheorem-1} here. }
\end{remark}

Now, we would like to introduce our estimates for the first non-zero
eigenvalue of the Wentzell eigenvalue problem of the weighted
Laplacian, which can be seen as a natural generalization of the
classical Steklov eigenvalue problem. However, before that, we need
to briefly mention several notions introduced by Bakry and \'{E}mery
\cite{be}.

Let $(N,\langle\cdot,\cdot\rangle)$ be an $(n+1)$-dimensional
compact oriented Riemannian manifold with smooth boundary $\partial
N$. The triple $(N,\langle\cdot,\cdot\rangle,e^{-\phi}dv)$ is called
a metric measure space (MMS for short), where $\phi\in
C^{\infty}(N)$ is a smooth function defined on $N$, and $dv$ is the
volume element. As introduced by  Bakry-\'{E}mery \cite{be}, the
weighted Laplacian (or the drifting Laplacian) and the
$K$-dimensional Bakry-\'{E}mery Ricci curvature can be defined as
follows
\begin{eqnarray*}
\Delta_{\phi}=\Delta-\langle\cdot,\nabla\phi\rangle
\end{eqnarray*}
and
\begin{eqnarray*}
\mathrm{Ric}_{\phi}^{K}=\mathrm{Ric} +\mathrm{Hess}\phi
-\frac{d\phi\otimes d\phi}{K-n-1},
\end{eqnarray*}
where, with the abuse of notations, $\Delta$ and $\nabla$ denote the
Laplace and the gradient operators on $N$ respectively, and
$\mathrm{Hess}$ is the Hessian operator on $N$ associated to the
metric $\langle\cdot,\cdot\rangle$. Here $K>n+1$ or $K=n+1$ if
$\phi$ is a constant function. When $K=\infty$, one can defined the
so-called $\infty$-dimensional Bakry-\'{E}mery Ricci curvature
(simply, \emph{Bakry-\'{E}mery Ricci curvature} or \emph{weighted
Ricci curvature}) as follows
\begin{eqnarray*}
\mathrm{Ric}_{\phi}=\mathrm{Ric} +\mathrm{Hess}\phi.
\end{eqnarray*}
By abuse of the notation, denote also by $\vec{\eta}$ the outward
unit normal vector field along the boundary $\partial N$. Let
$i:\partial N\hookrightarrow N$ be the standard inclusion. For
$X,Y\in\mathcal{H}(\partial N)$, i.e., the set of tangent vector
fields on $\partial N$, the second fundament form $II$ associated to
$\vec{\eta}$ is given by
$\langle\vec{\eta},\nabla_{X}Y\rangle=II(X,Y)$, and the mean
curvature $H=\frac{1}{n}\mathrm{tr}{II}$ of $\partial N$ is actually
the average of the trace of the second fundamental form. Naturally,
one can define the so-called \emph{weighted mean curvature} of the
inclusion $i$ as
$H_{\phi}=H-\frac{1}{n}\langle\vec{\eta},\nabla\phi\rangle$ on the
MMS $(N,\langle\cdot,\cdot\rangle,e^{-\phi}dv)$ - see, e.g.,
\cite{mg} for this notion.

On the compact MMS $(N,\langle\cdot,\cdot\rangle,e^{-\phi}dv)$
mentioned above, consider the following eigenvalue problem with the
Wentzell-type boundary condition
\begin{eqnarray} \label{WEP}
\left\{
    \begin{array}{ll}
      \Delta_{\phi}u=0,\qquad\quad & {\rm{in}}~N,\\
     -\beta\bar{\Delta}_{\phi}u+\frac{\partial u}{\partial\vec{\eta}}=\tau u,~\qquad\quad &
{\rm{on}}~\partial N,
    \end{array}
  \right.
\end{eqnarray}
where $\bar{\Delta}_{\phi}$ is the weighted Laplacian on the
boundary $\partial N$. In fact, (\ref{WEP}) is called \emph{the
Wentzell eigenvalue problem of the weighted Laplacian}. When
$\phi=const.$, i.e., the non-zero constant function, $\Delta_{\phi}$
and $\bar{\Delta}_{\phi}$ degenerate into usual Laplacians $\Delta$
on $N$ and $\bar{\Delta}$ on the boundary $\partial N$ respectively,
and (\ref{WEP}) becomes
\begin{eqnarray} \label{WEP-1}
\left\{
    \begin{array}{ll}
      \Delta u=0,\qquad\quad & {\rm{in}}~N,\\
     -\beta\bar{\Delta}u+\frac{\partial u}{\partial\vec{\eta}}=\tau u,~\qquad\quad &
{\rm{on}}~\partial N,
    \end{array}
  \right.
\end{eqnarray}
which, for $\beta\geq0$, has discrete spectrum and all the
eigenvalues in the spectrum can be listed increasingly as follows
\begin{eqnarray*}
0=\tau_{0}<\tau_{1}\leq\tau_{2}\leq\tau_{3}\leq\cdots\uparrow\infty.
\end{eqnarray*}
Recently, some interesting estimates for eigenvalues $\tau_{i}$ of
the eigenvalue problem (\ref{WEP-1}) have been obtained - see, e.g.,
\cite{DKL,dwx,dmwx,wx}. Especially, when $\beta=0$, (\ref{WEP-1})
degenerates into the Steklov eigenvalue problem (\ref{SEP}).

It is not difficult to find out that the weighted version
(\ref{WEP}) with $\beta\geq0$ also has discrete spectrum and all the
eigenvalues in the spectrum can be listed increasingly as follows
\begin{eqnarray*}
0=\tau_{0,\phi}<\tau_{1,\phi}\leq\tau_{2,\phi}\leq\tau_{3,\phi}\leq\cdots\uparrow\infty.
\end{eqnarray*}
 Besides, by the variational principle, it is easy to know that
the first non-zero eigenvalue $\tau_{1,\phi}$ of the eigenvalue
problem (\ref{WEP}) can be characterized as follows
\begin{eqnarray} \label{cha-1}
 \tau_{1,\phi}=\min\left\{\frac{\int_{N}|\nabla u|^{2}+\beta\int_{\partial N}|\bar{\nabla}u|^{2}}{\int_{\partial N}u^{2}}\Big{|}u\in W^{1,2}(N),u\neq0,\int_{\partial N}u=0\right\},
\end{eqnarray}
 where $\bar{\nabla}$ is the gradient operator on $\partial N$. Here, we would like to point out one thing that \emph{all integrations in (\ref{cha-1}) should have weighted volume elements, and for
 convenience, we have dropped them}. For $\tau_{1,\phi}$, we can obtain
its lower and sharp upper bounds under suitable assumptions on the
$K$-dimensional Bakry-\'{E}mery Ricci curvature
$\mathrm{Ric}^{K}_{\phi}$, the weighted mean curvature $H_{\phi}$,
and the second fundamental form $II$ of the boundary $\partial N$ -
see Theorems \ref{theorem3-1} and \ref{theorem3-2} for details.

The paper is organized as follows. Some preliminary facts will be
mentioned in Section 2. Proofs for Theorems \ref{maintheorem-1} and
 \ref{maintheorem-2} will be shown carefully in Section 3. In Section 4, estimates for the first non-zero eigenvalue of the Wentzell eigenvalue problem of the weighted
Laplacian will be investigated. An open problem will be issued in
the last section.

\section{Some useful facts}
\renewcommand{\thesection}{\arabic{section}}
\renewcommand{\theequation}{\thesection.\arabic{equation}}
\setcounter{equation}{0} \setcounter{maintheorem}{0}

In this section, some facts will be mentioned in the purpose of
proving the main conclusions of this paper. However, before that,
let us recall several useful facts about spherically symmetric
manifolds.

By making the radial sectional curvature upper bound assumption, Mao
\cite[Theorem 2.3.2]{jm1} (see also \cite[Theorem 4.2]{fmi}) has
proven the following Bishop-type volume comparison theorem.

\begin{theorem} (\cite{fmi,jm1}, generalized Bishop's volume comparison theorem II) \label{BishopII}
Assume that $M$ has a radial sectional curvature upper bound
$k(t)=-\frac{f''(t)}{f(t)}$ w.r.t. $p\in{M}$ for
$t<\alpha\leq\min\{inj_{c}(p),l\}$, where
$inj_{c}(p)=\inf_{\xi}c_{\xi}$, with  $\gamma_{\xi}(c_{\xi})$  a
first conjugate point along the geodesic
$\gamma_{\xi}(t)=\exp_{p}(t\xi)$. Then on $(0,\alpha)$,
\begin{eqnarray}  \label{4vc}
\left(\frac{\sqrt{|g|}}{f^{n-1}}\right)'\geq0, \quad\quad
\sqrt{|g|}(t,\xi)\geq{f^{n-1}(t)},
 \end{eqnarray}
and  equality occurs in the first inequality at $t_{0}\in(0,\alpha)$
if and only if
 \begin{eqnarray*}
\mathcal{R}=-\frac{f''(t)}{f(t)}, \quad
\mathbb{A}=f(t)\mathrm{I}_{(n-1)\times(n-1)},
 \end{eqnarray*}
 on all of $[0,t_{0}]$.
\end{theorem}

There are another three important facts for our model spaces, which
will be used later, we would like to list here.

\begin{itemize}

\item \textbf{Fact 1}: (\cite[page 706]{fmi}, \cite[page 379]{jm2}) By proposition 42 and corollary 43 of chapter 7 in \cite{k} or
subsection 3.2.3 of chapter 3 in \cite{p}, we know that the radial
sectional curvature, and the  component of the radial Ricci tensor
of the spherically symmetric manifold
$M^{\ast}=[0,l)\times_{f(t)}\mathbb{S}^{n-1}$ with the base point
$p$ are given by
\begin{equation} \label{radialcurvatures}
\begin{array}{ll}
\mathcal{K}(V,\frac{d}{dt})=R(\frac{d}{d t},V,\frac{d}{d t}, V)
=-\frac{f''(t)}{f(t)}&\mbox{~for}~~ V\in T_{\xi}\mathbb{S}^{n-1},
~|V|_{g}=1,\\
{\rm{Ric}}(\frac{d}{d t},\frac{d}{d t})=-(n-1)\frac{f''(t)}{f(t)}.
\end{array}
\end{equation}
Thus, Definition \ref{def1} (resp., Definition {\ref{def2}}) is
satisfied with equality in (\ref{2.4}) (resp., (\ref{2.5})) and
$k(t)=-f''(t)/f(t)$.

\item \textbf{Fact 2}: (\cite[Lemma 3]{e3})  Let $B_{r}$ be a
 ball in $\mathbb{R}^n$ endowed with a rotationally
invariant metric
\begin{eqnarray*}
dt^{2}+f^{2}(t)|d\xi|^{2},
\end{eqnarray*}
where, as before, $|d\xi|^{2}$ represents the round metric on the
unit sphere $\mathbb{S}^{n-1}\subseteq\mathbb{R}^{n}$. The first
nonconstant eigenfunction for the Steklov problem on $B_{r}$ has the
form
\begin{eqnarray*}
\varphi(t,\xi)=\psi(t)e(\xi)
\end{eqnarray*}
where $e(\xi)$ satisfies the equation $\Delta e+(n-1)e=0$ on
$\mathbb{S}^{n-1}$ and the function $\psi$ satisfies the
differential equation
\begin{equation*}
\begin{array}{ll}
\frac{1}{f^{n-1}(t)}\frac{d}{dt}\left(f^{n-1}(t)\frac{d}{dt}\psi(t)\right)-\frac{(n-1)\psi(t)}{f^{2}(t)}=0
&\mbox{~in}~~(0,r),\\
\psi'(r)=v_{1}(B_{r})\psi(r),~~~\psi(0)=0.
\end{array}
\end{equation*}

\item \textbf{Fact 3}: (\cite[Proposition 4]{e3}) Let $B_{r}$ be a
two-dimensional ball in $\mathbb{R}^2$ endowed with a rotationally
invariant metric
\begin{eqnarray*}
dt^{2}+f^{2}(t)|d\xi|^{2}.
\end{eqnarray*}
Then the first nonzero Steklov eigenvalue is $f^{-1}(r)$.

\end{itemize}

In order to get sharp bounds for the first non-zero eigenvalue
$\tau_{1,\phi}$ of the eigenvalue problem (\ref{WEP}) on the compact
MMS $(N,\langle\cdot,\cdot\rangle,e^{-\phi}dv)$ with boundary
$\partial N$, we need the following facts which have been proven in
\cite{bcp,bs}.

\begin{lemma} (\cite[Proposition
2.2]{bcp}) \label{IM-3} Assume that $u$ is a smooth function on
$N^{n+1}$, and  other same notations have the same meaning as those
at the end of Subsection \ref{sub1.2}. Then we have
\begin{eqnarray*}
|\mathrm{Hess}u|^{2}+\mathrm{Ric}_{\phi}(\nabla u,\nabla u)\geq
\frac{(\Delta_{\phi}u)^2}{K}+\mathrm{Ric}^{K}_{\phi}(\nabla u,\nabla
u)
\end{eqnarray*}
 for every $K>n+1$ or $K=n+1$ and $\phi$ is
constant. Moreover, equality holds if and only if $\mathrm{Hess}
u=\frac{\Delta u}{n+1}\langle\cdot,\cdot\rangle$ and $\langle\nabla
u,\nabla\phi\rangle=-\frac{K-n-1}{K}\Delta_{\phi}u$.
\end{lemma}

By directly applying Lemma \ref{IM-3} and the generalized Reilly
formula shown in \cite{md}, one can obtain the following fact.

\begin{lemma}  \label{IM-1}
Assume that $u$ is a smooth function on $N^{n+1}$, $h=\frac{\partial
u}{\partial\vec{\eta}}$, $z=u|_{\partial N}$, $\bar{\nabla}$ is the
gradient operator on $\partial N$, and other same notations have the
same meaning as those at the end of Subsection \ref{sub1.2}. Then we
have
\begin{eqnarray*}
\int_{N}\frac{K-1}{K}\left[(\Delta_{\phi}u)^2-\mathrm{Ric}_{\phi}^{K}(\nabla
u,\nabla u)\right]\geq\int_{\partial N}
\left[2h\bar{\Delta}_{\phi}z+nH_{\phi} h^{2}+II(\bar{\nabla}
z,\bar{\nabla} z)\right]
\end{eqnarray*}
 for every $K>n+1$ or $K=n+1$ and $\phi$ is
constant. Moreover, equality holds if and only if $\mathrm{Hess}
u=\frac{\Delta u}{n+1}\langle\cdot,\cdot\rangle$ and $\langle\nabla
u,\nabla\phi\rangle=-\frac{K-n-1}{K}\Delta_{\phi}u$.
\end{lemma}

\begin{remark}
\rm{ The first conclusion of Lemma \ref{IM-1} is actually (2.5) in
\cite{bs}. }
\end{remark}

Batista and Santos have proved the following conclusion, which is a
slight modification to \cite[Theorem 1.6]{hr}.

\begin{lemma} (\cite[Proposition 2.2]{bs}) \label{IM-2}
Let $N^{n+1}$ be a compact weighted Riemannian manifold with
nonempty boundary $\partial N$ and $\mathrm{Ric}^{K}_{\phi}\geq 0$.
If the second fundamental form of $\partial N$ satisfies $II\geq
c\mathrm{I}_{n\times n},$ in the quadratic form sense,  and
$H_{\phi}\geq \frac{K-1}{n}c,$
 then
 \begin{eqnarray*}
 \lambda_{1}^{c}(\partial N)\geq(K-1)c^2,
 \end{eqnarray*}
 where $\lambda_{1}^{c}$ is the first non-zero closed eigenvalue of the drifting Laplacian acting on functions on $\partial
 N$. The equality holds if and only if $N^{n+1}$ is isometric to a Euclidean ball
of radius $\frac{1}{c}$, $\phi$ is constant and $K=n+1$.
\end{lemma}

\section{Proofs of
the Steklov eigenvalue comparisons} \label{PSEC}
\renewcommand{\thesection}{\arabic{section}}
\renewcommand{\theequation}{\thesection.\arabic{equation}}
\setcounter{equation}{0} \setcounter{maintheorem}{0}

In this section, with the help of results mentioned in Section 2, by
using a similar method to that shown by Escobar \cite{e3}, we can
give the proofs of Theorems \ref{maintheorem-1} and
\ref{maintheorem-2} as follows.

\vspace{3mm}

 \textbf{\emph{Proof of Theorem
\ref{maintheorem-2}}}. Let $\psi(t)e(\xi)$ be the eigenfunction for
the first non-zero eigenvalue $v_{1}(\mathcal{B}_{n}(p^{+},r))$ on
$\mathcal{B}_{n}(p^{+},r)$ given as \textbf{Fact 2}. By \textbf{Fact
2} also, we have
\begin{eqnarray*}
\psi'(t)=\frac{n-1}{f^{n-1}(t)}\int_{0}^{t}\psi(s)f^{n-3}(s)ds,
\end{eqnarray*}
which implies that $\psi$ is positive (resp., negative) in $(0,r)$
if $\psi'(t)$ is positive (resp., negative) in a neighborhood of
zero. So, the function $\psi$ determined in \textbf{Fact 2} does not
change sign in $(0,r)$. Without loss of generality, we can assume
$\psi\geq0$. Correspondingly, $\psi'(t)>0$.

Consider the test function $\varphi(t,\xi)=a_{+}(t)e_{1}(\xi)$,
where $e_{1}(\xi)$ is an eigenfunction of the first non-zero closed
eigenvalue $\lambda_{1}^{c}(\partial B(p,r))$ of the Laplacian on
the boundary $\partial B(p,r)$, and
\begin{eqnarray*}
a_{+}(t):=\max\{a(t),0\},
\end{eqnarray*}
\begin{eqnarray*}
a(t):=\psi(t)\left[\frac{f^{n-1}(t)}{h(t)}\right]^{1/2}+\int_{t}^{r}\psi(s)\left(\left[\frac{f^{n-1}(s)}{h(s)}\right]^{1/2}\right)'ds,
\end{eqnarray*}
with
$h(t):=\max\left\{d^{\ast}(t),\frac{f^{2}(t)}{n-1}d^{\sharp}(t)\right\}$
and
\begin{eqnarray*}
 d^{\sharp}(t)=\int_{\mathbb{S}^{n-1}}|\nabla
e_{1}|^{2}_{\mathbb{S}^{n-1}}(\xi)J^{n-3}(t,\xi),
 \end{eqnarray*}
\begin{eqnarray*}
  d^{\ast}(t)=\int_{\mathbb{S}^{n-1}}
e_{1}^{2}(\xi)\cdot\det\mathbb{A}(t,\xi)=\int_{\mathbb{S}^{n-1}}
e_{1}^{2}(\xi)\sqrt{|g|}(t,\xi)=\int_{\mathbb{S}^{n-1}}
e_{1}^{2}(\xi)J^{n-1}(t,\xi).
\end{eqnarray*}
Clearly, $h(t)$ is Lipschitz continuous and hence differentiable
almost everywhere. Besides, by Rayleigh's theorem and Max-min
principle, we have
\begin{eqnarray} \label{3-1}
\int_{\partial B(p,r)}|\nabla e_{1}|^{2}=\lambda_{1}^{c}(\partial
B(p,r))\cdot\int_{\partial B(p,r)}e_{1}^{2}
\end{eqnarray}
and
\begin{eqnarray}  \label{3-2}
\int_{\partial B(p,r)}e_{1}(\xi)=0.
\end{eqnarray}

By Theorem \ref{BishopII}, we have
\begin{eqnarray*}
\left(\frac{\sqrt{|g|}}{f^{n-1}}\right)'=\frac{(n-1)J^{n-2}}{f^{n}}(J'f-f'J)\geq0
\end{eqnarray*}
in $(0,r)$, which implies
\begin{eqnarray} \label{VCS}
f'J-J'f\leq0
\end{eqnarray}
in  $(0,r)$.

Now, we claim that
\begin{eqnarray} \label{claim-1}
\left[\frac{f^{n-1}(t)}{h(t)}\right]'\leq0
\end{eqnarray}
in $(0,r)$. This is because if $h(t)=d^{\ast}(t)$ in a neighborhood
of some point in $(0,r)$, then together with (\ref{VCS}), we have
\begin{eqnarray*}
h^{2}(t)\left[\frac{f^{n-1}(t)}{h(t)}\right]'=(n-1)d^{\ast}f^{n-2}f'-f^{n-1}(t)\cdot\left(d^{\ast}\right)'=(n-1)f^{n-2}\int_{\mathbb{S}^{n-1}}e_{1}^{2}J^{n-2}(f'J-J'f)\leq0.
\end{eqnarray*}
If $h(t)=\frac{f^{2}(t)}{n-1}d^{\sharp}(t)$ in a neighborhood of
some point in $(0,r)$, then together with (\ref{VCS}), we have
\begin{eqnarray*}
\left(d^{\sharp}(t)\right)^{2}\cdot\left[\frac{f^{n-3}(t)}{d^{\sharp}(t)}\right]'&=&(n-3)d^{\sharp}f^{n-4}f'-f^{n-3}(t)\cdot\left(d^{\sharp}\right)'\\
&=&(n-3)f^{n-4}\int_{\mathbb{S}^{n-1}}|\nabla
e_{1}|^{2}_{\mathbb{S}^{n-1}}J^{n-4}(f'J-J'f)\leq0,
\end{eqnarray*}
which implies $\left[\frac{f^{n-1}(t)}{h(t)}\right]'\leq0$.
Therefore, our claim (\ref{claim-1}) is true. So, the function
\begin{eqnarray*}
\int_{t}^{r}\psi(s)\left(\left[\frac{f^{n-1}(s)}{h(s)}\right]^{1/2}\right)'ds
\end{eqnarray*}
is negative for $t<r$, except when $J(t,\xi)=f(t)$. Together with
the definitions of $a(t)$ and $a_{+}(t)$, it follows that
\begin{eqnarray*}
a(t)\leq
a_{+}(t)\leq\psi(t)\left[\frac{f^{n-1}(t)}{h(t)}\right]^{1/2}
\end{eqnarray*}
on $(0,r)$. Using the above inequality, definitions of $h(t)$ and
$d^{\sharp}(t)$, one can get
\begin{eqnarray*}
\int_{B(p,r)}a_{+}^{2}(t)|\nabla
e_{1}|^{2}_{\mathbb{S}^{n-1}}J^{n-3}(t,\xi)dtd\sigma&=&\int_{0}^{r}a_{+}^{2}(t)\left(\int_{\mathbb{S}^{n-1}}|\nabla
e_{1}|^{2}_{\mathbb{S}^{n-1}}J^{n-3}(t,\xi)d\sigma\right)dt\\
&\leq&\int_{0}^{r}\psi^{2}\frac{f^{n-1}(t)d^{\sharp}(t)}{h(t)}dt\\
&\leq&(n-1)\int_{0}^{r}\psi^{2}f^{n-3}(t)dt,
\end{eqnarray*}
where, as before, $d\sigma$ denotes the $(n-1)$-dimensional volume
element on $\mathbb{S}^{n-1}$. Therefore, we have
\begin{eqnarray} \label{3-5}
\int_{B(p,r)}|\nabla\varphi|^{2}&=&\int_{B(p,r)}\left(a_{+}'(t)\right)^{2}e^{2}_{1}(\xi)J^{n-1}(t,\xi)dtd\sigma \nonumber\\
&&\qquad + \int_{B(p,r)}a_{+}^{2}(t)|\nabla
e_{1}|^{2}_{\mathbb{S}^{n-1}}J^{n-3}(t,\xi)dtd\sigma \nonumber\\
&\leq&\int_{B(p,r)}\left(\psi'(t)\right)^{2}\frac{f^{n-1}(t)}{h(t)}e^{2}_{1}(\xi)J^{n-1}(t,\xi)dtd\sigma+(n-1)\int_{0}^{r}\psi^{2}f^{n-3}(t)dt \nonumber\\
&=&\int_{0}^{r}\left(\psi'(t)\right)^{2}\left(\frac{\int_{\mathbb{S}^{n-1}}e^{2}_{1}(\xi)J^{n-1}(t,\xi)d\sigma}{h(t)}\right)f^{n-1}(t)dt+(n-1)\int_{0}^{r}\psi^{2}f^{n-3}(t)dt \nonumber\\
&\leq&\int_{0}^{r}\left(\psi'(t)\right)^{2}f^{n-1}(t)dt+(n-1)\int_{0}^{r}\psi^{2}f^{n-3}(t)dt.
\end{eqnarray}
On the other hand, by direct calculation, we have
\begin{eqnarray} \label{3-6}
\int_{\partial
B(p,r)}\varphi^{2}=a_{+}^{2}(r)\int_{\mathbb{S}^{n-1}}e_{1}^{2}(\xi)J^{n-1}(r,\xi)d\sigma=\frac{\psi^{2}(r)f^{n-1}(r)d^{\ast}(r)}{h(r)}.
\end{eqnarray}
Equality (\ref{3-1}) is equivalent to
\begin{eqnarray*}
\lambda_{1}^{c}(\partial B(p,r)) =\frac{d^{\sharp}(r)}{d^{\ast}(r)}.
\end{eqnarray*}
Together with the assumption (\ref{b1}), we have
\begin{eqnarray*}
d^{\sharp}(r)\leq d^{\ast}(r)\frac{n-1}{f^{2}(r)},
\end{eqnarray*}
which means $\frac{f^{2}(t)}{n-1}d^{\sharp}(t)\leq d^{\ast}(r)$. So,
one has $h(r)=d^{\ast}(r)$, and then (\ref{3-6}) becomes
\begin{eqnarray}  \label{3-7}
\int_{\partial B(p,r)}\varphi^{2}=\psi^{2}(r)f^{n-1}(r).
\end{eqnarray}
Combining (\ref{3-5}) and (\ref{3-7}) yields
\begin{eqnarray} \label{3-8}
\frac{\int_{B(p,r)}|\nabla\varphi|^{2}}{\int_{\partial
B(p,r)}\varphi^{2}}\leq\frac{\int_{0}^{r}\left(\psi'(t)\right)^{2}f^{n-1}(t)dt+(n-1)\int_{0}^{r}\psi^{2}f^{n-3}(t)dt}{\psi^{2}(r)f^{n-1}(r)}.
\end{eqnarray}
Since by (\ref{3-2}), the equality
\begin{eqnarray*}
\int_{\partial B(p,r)}\varphi=\psi(r)\int_{\partial
B(p,r)}e_{1}(\xi)=0
\end{eqnarray*}
holds, and then by applying \textbf{Fact 2}, the characterization
(\ref{cha}) and (\ref{3-8}), we can obtain
\begin{eqnarray*}
v_{1}(B(p,r))\leq
\frac{\int_{B(p,r)}|\nabla\varphi|^{2}}{\int_{\partial
B(p,r)}\varphi^{2}}\leq\frac{\int_{0}^{r}\left(\psi'(t)\right)^{2}f^{n-1}(t)dt+(n-1)\int_{0}^{r}\psi^{2}f^{n-3}(t)dt}{\psi^{2}(r)f^{n-1}(r)}=v_{1}(\mathcal{B}_{n}(p^{+},r)),
\end{eqnarray*}
which is exactly (\ref{e3}). It is easy to find that the equality in
the above inequality holds if and only if $f(t)=J(t,\xi)$, which, by
Theorem \ref{BishopII}, Definition \ref{defsp} and \textbf{Fact 1},
is equivalent to say that $B(p,r)$ is isometric to
$\mathcal{B}_{n}(p^{+},r)$. This completes the proof. \hfill
$\square$

\vspace{2mm}

Now, in order to prove Theorem \ref{maintheorem-1}, we only need to
confirm that the precondition (\ref{b1}) holds naturally for the
case of two and three dimensions.

\vspace{3mm}

 \textbf{\emph{Proof of Theorem
\ref{maintheorem-1}}}.  \label{ProofTH1.5}  We divide the proof into
two cases as follows:

Case 1. If $n=2$, i.e., $M$ is a $2$-dimensional Riemannian surface,
then $\partial B(p,r)$ is connected and diffeomorphic to
$\mathbb{S}^{1}$ because $r<\mathrm{inj}(p)$. Therefore, by
Weinstock's Theorem \cite{w1}, we have
\begin{eqnarray} \label{3-10-1}
v_{1}(B(p,r))\leq\frac{2\pi}{L}=\frac{2\pi}{\int_{0}^{2\pi}J(t,\xi)d\sigma}.
\end{eqnarray}
Here $L$ represents the perimeter of the closed curve $\partial
B(p,r)$, which is diffeomorphic to $\mathbb{S}^{1}$. Applying
Theorem \ref{BishopII}, from which one has $J(t,\xi)\geq f(t)$ on
$(0,r)$, and \textbf{Fact 3} to (\ref{3-10-1}), we can get
\begin{eqnarray*}
v_{1}(B(p,r))\leq\frac{2\pi}{\int_{0}^{2\pi}J(t,\xi)d\sigma}\leq\frac{2\pi}{\int_{0}^{2\pi}f(t)d\sigma}=\frac{1}{f(t)}=v_{1}(\mathcal{B}_{n}(p^{+},r)).
\end{eqnarray*}
Obviously, equality in the above inequality holds if and only if
$f(t)=J(t,\xi)$, which, by Theorem \ref{BishopII}, Definition
\ref{defsp} and \textbf{Fact 1}, is equivalent to say that $B(p,r)$
is isometric to $\mathcal{B}_{n}(p^{+},r)$.

Case 2. If $n=3$, i.e., $M$ is a  Riemannian $3$-manifold, then
$\partial B(p,r)$ is connected and diffeomorphic to $\mathbb{S}^{2}$
because $r<\mathrm{inj}(p)$. Let $e$ be any eigenfunction of the
first non-zero closed eigenvalue $n-1=2$ of the Laplacian on
$\mathbb{S}^{2}$. Therefore, $e=\langle\zeta,x\rangle$, where
$\zeta\in\mathbb{S}^{2}$, and $x$ represents Euclidean coordinates.
By direct calculation, we have
\begin{eqnarray} \label{3-10}
\frac{\int_{\partial B(p,r)}|\nabla e|^2}{\int_{\partial
B(p,r)}e^2}&=&\frac{\int_{\mathbb{S}^{n-1}}|\nabla
e|^{2}_{\mathbb{S}^{n-1}}J^{n-3}(r,\xi)d\sigma}{\int_{\mathbb{S}^{n-1}}e^{2}J^{n-1}(r,\xi)d\sigma} \nonumber\\
&=& \frac{\int_{\mathbb{S}^{2}}|\nabla
e|^{2}_{\mathbb{S}^{2}}d\sigma}{\int_{\mathbb{S}^{2}}e^{2}J^{2}(r,\xi)d\sigma}\nonumber\\
&\leq&\frac{1}{f^{2}(r)}\frac{\int_{\mathbb{S}^{2}}|\nabla
e|^{2}_{\mathbb{S}^{2}}d\sigma}{\int_{\mathbb{S}^{2}}e^{2}d\sigma}\nonumber\\
&=&\frac{2}{f^{2}(r)}=\lambda_{1}^{c}(\partial\mathcal{B}_{n}(p^{+},r)),
\end{eqnarray}
where the last inequality holds because $J(t,\xi)\geq f(t)$ on
$(0,r)$. On the other hand, define
$F:\mathbb{S}^{2}\rightarrow\mathbb{R}^2$ by
$F(e_{\zeta})=\int_{\partial
B(p,r)}e_{\zeta}=\int_{\mathbb{S}^2}e_{\zeta}J^{2}(r,\xi)d\sigma$,
and then we know that there must exist some $e\in\mathbb{S}^2$ such
that $F(e)=0$, i.e.,
\begin{eqnarray} \label{3-11}
\int_{\partial B(p,r)}e=\int_{\mathbb{S}^2}e\cdot
J^{2}(r,\xi)d\sigma.
\end{eqnarray}
This is because the function $F$ is continuous and
$F(e_{\zeta})=-F(-e_{\zeta})$ by choosing antipodal points. By
(\ref{3-11}), Rayleigh's theorem and Max-min principle,  one has
 \begin{eqnarray} \label{3-12}
 \lambda_{1}^{c}(B(p,r))\leq\frac{\int_{\partial B(p,r)}|\nabla e|^2}{\int_{\partial
B(p,r)}e^2}.
 \end{eqnarray}
Combining (\ref{3-10}) and (\ref{3-12}) results in
\begin{eqnarray*}
\lambda_{1}^{c}(B(p,r))\leq\lambda_{1}^{c}(\partial\mathcal{B}_{n}(p^{+},r)),
\end{eqnarray*}
and then by applying Theorem \ref{maintheorem-2} directly, we have
$v_{1}(B(p,r))\leq v_{1}(\mathcal{B}_{n}(p^{+},r))$. Clearly, the
rigidity conclusion for the equality $v_{1}(B(p,r))=
v_{1}(\mathcal{B}_{n}(p^{+},r))$ here can be attained by using
almost the same argument as the $2$-dimensional case.

This completes the proof. \hfill $\square$

\begin{remark}
\rm{It is easy to check that the method used in Case 2 here is also
valid for Case 1.}
\end{remark}

\section{Some eigenvalue estimates}
\renewcommand{\thesection}{\arabic{section}}
\renewcommand{\theequation}{\thesection.\arabic{equation}}
\setcounter{equation}{0} \setcounter{maintheorem}{0}

Estimates for the first non-zero eigenvalue $\tau_{1,\phi}$ of the
eigenvalue problem (\ref{WEP}) with $\beta\geq0$ will be given in
this section. In fact, we can obtain the followings.

\begin{theorem} \label{theorem3-1}
Assume that $(N,\langle\cdot,\cdot\rangle,e^{-\phi}dv)$ is an
$(n+1)$-dimensional compact connected MMS with smooth boundary
$\partial N$. If $\mathrm{Ric}_{\phi}^{K}\geq0$ and
$H_{\phi}\geq\frac{(K-1)c}{n}$, $II\geq c\cdot\mathrm{I}_{n\times
n}$ for some positive constant $c>0$, with $II$ the second
fundamental form of $\partial N$, then, for the eigenvalue problem
(\ref{WEP}) with $\beta\geq0$, we have
\begin{eqnarray} \label{TC4-1}
\tau_{1,\phi}\leq\beta\lambda_{1}^{c}+\frac{\sqrt{\lambda_{1}^{c}}\left[\sqrt{\lambda_{1}^{c}}+\sqrt{\lambda_{1}^{c}-(K-1)c^2}\right]}{(K-1)c},
\end{eqnarray}
where, similar as before,  $\mathrm{I}_{n\times n}$ represents the
$n\times n$ identity matrix, and $\lambda_{1}^{c}$ is the first
non-zero closed eigenvalue of the drifting Laplacian on the boundary
$\partial N$. Equality in (\ref{TC4-1}) holds if and only if  $N$ is
isometric to an $(n+1)$-dimensional Euclidean ball of radius
$\frac{1}{c}$, $\phi$ is the non-zero constant function, and
$K=n+1$.
\end{theorem}

\begin{proof}
Let $u$ be the solution to the following problem
\begin{eqnarray*}
\left\{
\begin{array}{ll}
\Delta_{\phi}u=0  \qquad\quad & {\rm{in}}~N,\\
u=z \qquad\quad & {\rm{on}}~\partial N,
\end{array}
\right.
\end{eqnarray*}
where $z$ is the eigenfunction of the first non-zero closed
eigenvalue $\lambda_{1}^{c}$ of the drifting Laplacian
$\bar{\Delta}_{\phi}$ on $\partial N$. That is,
$\bar{\Delta}_{\phi}z+\lambda_{1}^{c}z=0$ on $\partial N$,
$\int_{\partial N}z=0$. Set $h=\frac{\partial
u}{\partial\vec{\eta}}$. By (\ref{cha-1}) and the fact that
$\int_{\partial N}z=\int_{\partial N}u=0$, we have
 \begin{eqnarray} \label{4-2}
 \tau_{1,\phi}&\leq&\frac{\int_{N}|\nabla u|^2+\beta\int_{\partial N} |\bar{\nabla}z|^2}{\int_{\partial
 N}z^2} \nonumber\\
&=&\frac{\int_{\partial N}zh+\beta\int_{\partial N}
|\bar{\nabla}z|^2}{\int_{\partial
 N}z^2} \nonumber\\
 &=&\beta\lambda_{1}^{c}+\frac{\int_{\partial N}zh}{\int_{\partial
 N}z^2}
 \end{eqnarray}
 where the first equality holds by the usage of the divergence
 theorem. By Lemmas \ref{IM-3} and \ref{IM-1}, together with
 assumptions $\mathrm{Ric}_{\phi}^{K}\geq0$, $H_{\phi}\geq\frac{(K-1)c}{n}$, $II\geq
c\cdot\mathrm{I}_{n\times n}$ for some positive constant $c>0$, we
have
\begin{eqnarray*}
0\
&\geq&\int_{N}\frac{K-1}{K}\left[(\Delta_{\phi}u)^2-\mathrm{Ric}^{K}_{\phi}(\nabla u,\nabla u)\right]\\
&\geq&\int_{\partial N}\left[nH_{\phi}h^{2}+2h\bar{\Delta}_{\phi}z+II(\bar{\nabla}z,\bar{\nabla}z)\right]\\
&\geq&\int_{\partial N}\left[(K-1)ch^{2}-2\lambda_{1}^{c}zh+c|\bar{\nabla}z|^2\right]\\
&=&\int_{\partial
N}\left[(K-1)ch^{2}-2\lambda_{1}^{c}zh+c\lambda_{1}^{c}z^2\right],
\end{eqnarray*}
and then, by applying H\"{o}lder's inequality to the above
inequality, it follows that
\begin{eqnarray*}
0\
&\geq&\int_{\partial N}\left[(K-1)ch^{2}-2\lambda_{1}^{c}hz+c\lambda_{1}^{c}z^{2}\right]\\
&\geq&\int_{\partial N}(K-1)ch^2-2\lambda_{1}^{c}\left(\int_{\partial N}h^{2}\right)^\frac{1}{2}\left(\int_{\partial N}z^2\right)^\frac{1}{2}+\int_{\partial N}c\lambda_{1}^{c}z^{2}\\
&=&(K-1)c\left[\left(\int_{\partial
N}h^{2}\right)^\frac{1}{2}-\frac{\lambda_{1}^{c}}{(K-1)c}\left(\int_{\partial
N}z^2\right)^\frac{1}{2}\right]^2+
\left[c\lambda_{1}^{c}-\frac{(\lambda_{1}^{c})^2}{(K-1)c}\right]\int_{\partial
N}z^2,
\end{eqnarray*}
which is equivalent with
\begin{eqnarray} \label{4-3}
\left(\int_{\partial
N}h^2\right)^\frac{1}{2}\leq\frac{\sqrt{\lambda_{1}^{c}}\left[\sqrt{\lambda_{1}^{c}}+\sqrt{\lambda_{1}^{c}-(K-1)c^2}\right]}{(K-1)c}
\left(\int_{\partial N}z^2\right)^\frac{1}{2}.
\end{eqnarray}
Combining (\ref{4-2}) and (\ref{4-3}) yields
\begin{eqnarray*}
\tau_{1,\phi}\
&\leq&\beta\lambda_{1}^{c}+\frac{\int_{\partial N}zh}{\int_{N}z^2} \\
&\leq&\beta\lambda_{1}^{c}+\frac{\left(\int_{\partial N}h^2\right)^\frac{1}{2}}{\left(\int_{\partial N}z^2\right)^\frac{1}{2}}\\
&\leq&\beta\lambda_{1}^{c}+\frac{\sqrt{\lambda_{1}^{c}}\left[\sqrt{\lambda_{1}^{c}}+\sqrt{\lambda_{1}^{c}-(K-1)c^2}\right]}{(K-1)c},
\end{eqnarray*}
which is exactly (\ref{TC4-1}). If equality holds in (\ref{TC4-1}),
then all inequalities become equalities, and through the above
argument, one has $\mathrm{Hess}u=0$ and
\begin{eqnarray*}
h=\frac{\sqrt{\lambda_{1}^{c}}\left[\sqrt{\lambda_{1}^{c}}+\sqrt{\lambda_{1}^{c}-(K-1)c^2}\right]}{(K-1)c}z.
\end{eqnarray*}
Therefore, taking a local orthonormal fields $\{e_{i}\}^n_{i=1}$
tangent to $\partial N$, similar to the calculation in \cite[page
9]{bs}, we can obtain
\begin{eqnarray*}
0\
&=&\sum\limits^{n}_{i=1}\mathrm{Hess}u(e_{i},e_{i})\\
&=&\sum\limits^{n}_{i=1}\langle\nabla_{e_{i}}\nabla u,e_{i}\rangle\\
&=&\bar{\Delta}_{\phi}z+nH_{\phi}h\\
&=&-\lambda_{1}^{c}z+c(K-1)h\\
&=&-\lambda_{1}^{c}z+c(K-1)\frac{\sqrt{\lambda_{1}^{c}}\left[\sqrt{\lambda_{1}^{c}}+\sqrt{\lambda_{1}^{c}-(K-1)c^2}\right]}{(K-1)c}z,
\end{eqnarray*}
which implies $\lambda_{1}^{c}=(K-1)c^2$. Then, under the
assumptions for $\mathrm{Ric}^{K}_{\phi}$, $II$ and $H_{\phi}$, by
Lemma \ref{IM-2}, we know that $N$ is isometric to an
$(n+1)$-dimensional Euclidean ball of radius $\frac{1}{c}$, $\phi$
is a non-zero constant function and $K=n+1$. This completes the
proof.
\end{proof}

A lower bound for $\tau_{1,\phi}$ can also be obtained as follows.

\begin{theorem} \label{theorem3-2}
Assume that $(N,\langle\cdot,\cdot\rangle,e^{-\phi}dv)$ is an
$(n+1)$-dimensional compact connected MMS with smooth boundary
$\partial N$. If $\mathrm{Ric}_{\phi}^{K}\geq0$,
$H_{\phi}\geq\frac{(K-1)c}{n}$, $II\geq c\cdot\mathrm{I}_{n\times
n}$ for some positive constant $c>0$,  then, for the eigenvalue
problem (\ref{WEP}) with $\beta\geq0$, we have
\begin{eqnarray}  \label{TC4-2}
\tau_{1,\phi}>\frac{1}{2}c\left[1+(K-1)c\beta+\sqrt{(K-1)c^2\beta^2+2(K-1)c\beta}\right].
\end{eqnarray}
\end{theorem}

\begin{proof}
We use a similar proof to that of \cite[Theorem 1.3]{wx}. In order
to state precisely, we divide the proof into two cases as follows:

Case 1. Assume that $\beta=0$.  Let $u$ be an eigenfunction of the
first non-zero eigenvalue $\tau_{1,\phi}$ of the eigenvalue problem
(\ref{WEP}) with $\beta=0$. Set $h=\frac{\partial
u}{\partial\vec{\eta}}|_{\partial N}$, $z=u|_{\partial N}$. By
Lemmas \ref{IM-3} and \ref{IM-1}, together with the assumptions
$\mathrm{Ric}_{\phi}^{K}\geq0$, $H_{\phi}\geq\frac{(K-1)c}{n}$,
$II\geq c\cdot\mathrm{I}_{n\times n}$ for some positive constant
$c>0$, we have
\begin{eqnarray*}
0\
&\geq&\int_{\partial N}\frac{K-1}{K}\left[(\Delta_{\phi}u)^{2}-\mathrm{Ric}^{K}_{\phi}(\nabla u,\nabla u)\right]\\
&\geq&\int_{\partial N}\left[nH_{\phi}h^2+2h\bar{\Delta}_{\phi}z+II(\bar{\nabla}z,\bar{\nabla}z)\right]\\
&>&\int_{\partial N}\left[-2\langle\bar{\nabla} z,\bar{\nabla} h\rangle+c\langle\bar{\nabla}z,\bar{\nabla}z\rangle\right]\\
&=&\int_{\partial
N}\left[-2\tau_{1,\phi}\langle\bar{\nabla}z,\bar{\nabla}z\rangle+c\langle\bar{\nabla}z,\bar{\nabla}z\rangle\right],
\end{eqnarray*}
which implies
 \begin{eqnarray}  \label{4-5}
 \tau_{1,\phi}>\frac{c}{2}.
 \end{eqnarray}

Case 2. Assume that $\beta>0$. Let $u$ be an eigenfunction of the
first non-zero eigenvalue $\tau_{1,\phi}$ of the eigenvalue problem
(\ref{WEP}) with $\beta>0$. Set $\gamma=\frac{1}{\beta}$,
$\tau=\frac{\tau_{1,\phi}}{\beta}$, $h=\frac{\partial
u}{\partial\vec{\eta}}|_{\partial N}$, $z=u|_{\partial N}$. By
(\ref{WEP}), one has
\begin{eqnarray} \label{4-6}
\Delta_{\phi}u=0, \qquad \bar{\Delta}_{\phi}z=\gamma h-\tau z.
\end{eqnarray}
By Lemmas \ref{IM-3} and \ref{IM-1}, together with the assumption
$\mathrm{Ric}_{\phi}^{K}\geq0$, we have
\begin{eqnarray}  \label{4-7}
0\geq\int_{N}\frac{K-1}{K}\left[(\Delta_{\phi}u)^2-\mathrm{Ric}^{K}_{\phi}(\nabla
u,\nabla u)\right]\geq\int_{\partial
N}\left[nH_{\phi}h^2+2h\bar{\Delta}_{\phi}z+II(\bar{\nabla}z,\bar{\nabla}z)\right],
\end{eqnarray}
which, together with assumptions $H_{\phi}\geq\frac{(K-1)c}{n}$,
$II\geq c\cdot\mathrm{I}_{n\times n}$ for some positive constant
$c>0$, implies
\begin{eqnarray} \label{4-8}
0\
&\geq&\int_{\partial N}\left[(K-1)ch^{2}+2h(\gamma h-\tau z)+c|\bar{\nabla}z|^2\right] \nonumber\\
&=&\int_{\partial N}\left[(K-1)ch^{2}+2\gamma h^2-2\tau hz-cz\bar{\Delta}_{\phi}z \right]\nonumber\\
&=&\int_{\partial N}\left[c\tau z^2-(2\tau+c\gamma)hz+\left((K-1)c+2\gamma\right)h^2\right] \nonumber\\
&=&\left[(K-1)c+2\gamma\right]\int_{\partial
N}\left[h-\frac{(\tau+\frac{c\gamma}{2})z}{(K-1)c+2\gamma}\right]^{2}+\left[c\tau-\frac{\left(\tau+\frac{c\gamma}{2}\right)^2}{(K-1)c+2\gamma}\right]\int_{\partial
N}z^2.
\end{eqnarray}
Therefore, from the above inequality, one has
\begin{eqnarray*}
\left[\frac{(\tau+\frac{c\gamma}{2})^2}{(K-1)c+2\gamma}-c\tau\right]\int_{\partial
N}z^2\geq\left[(K-1)c+2\gamma\right]\int_{\partial N}\left[h-
\frac{(\tau+\frac{c\gamma }{2})z}{(K-1)c+2\gamma}\right]^2,
\end{eqnarray*}
which implies
\begin{eqnarray*}
\frac{(\tau+\frac{c\gamma}{2})^2}{(K-1)c+2\gamma}-c\tau\geq0.
\end{eqnarray*}
Solving the above inequality yields
\begin{eqnarray} \label{4-9}
\tau\geq\frac{\left[\sqrt{(K-1)^2c^2+2c(K-1)\gamma}+\gamma+(K-1)c\right]c}{2}
\end{eqnarray}
or
\begin{eqnarray} \label{4-10}
\tau\leq\frac{\left[\gamma+(K-1)c-\sqrt{(K-1)^2c^2+2c(K-1)\gamma}\right]c}{2}.
\end{eqnarray}
Now, we show that (\ref{4-10}) cannot happen. In fact, multiplying
the second equation of (\ref{4-6}) by $z$ and integrating over
$\partial N$, one can obtain
 \begin{eqnarray} \label{4-11}
\tau\int_{\partial N}z^{2}&=&\int_{\partial
N}|\bar{\nabla}z|^{2}+\gamma\int_{\partial N}hz \nonumber\\
&=& \int_{\partial N}|\bar{\nabla}z|^{2}+\gamma\int_{N}|\nabla
u|^{2},
 \end{eqnarray}
where the fact $\Delta_{\phi}u=0$ in $N$ has been used. Since
$\int_{\partial N}z=0$, $z\neq0$, by Lemma \ref{IM-2}, it follows
that
\begin{eqnarray} \label{4-12}
 \int_{\partial N}|\bar{\nabla}z|^{2}\geq\lambda^{c}_{1}(\partial N)\int_{\partial N}z^{2}\geq(K-1)c^{2}\int_{\partial N}z^{2}.
\end{eqnarray}
Combining (\ref{4-11}) and (\ref{4-12}), together with the fact
$\int_{N}|\nabla u|^{2}>\frac{c}{2}\int_{\partial N}z^{2}$, yields
 \begin{eqnarray*}
 \tau>(K-1)c^{2}+\frac{c\gamma}{2},
 \end{eqnarray*}
which implies that (\ref{4-10}) cannot hold. In order to get the
estimate (\ref{TC4-2}), one only needs to exclude the equality case
in (\ref{4-9}). We shall get this fact by contradiction. Suppose
that
\begin{eqnarray} \label{4-13}
\tau=\frac{\left[\sqrt{(K-1)^2c^2+2c(K-1)\gamma}+\gamma+(K-1)c\right]c}{2}.
\end{eqnarray}
From the above argument, we know that if (\ref{4-13}) holds, then
all inequalities in (\ref{4-7}) and (\ref{4-8}) should take equality
sign. Therefore, it follows that
\begin{eqnarray} \label{4-14}
h=\frac{(\tau+\frac{c\gamma}{2})z}{(K-1)c+2\gamma} \qquad
\mathrm{on}~\partial N,
\end{eqnarray}
 $II=c\mathrm{I}_{n\times n}$, i.e., all principal curvatures of $\partial N$ equal $c$, and, by
Lemma \ref{IM-1} and the first equation of (\ref{4-6}),
\begin{eqnarray} \label{4-15}
\mathrm{Hess}u=\frac{\Delta_{\phi}u-\frac{K-n-1}{K}\Delta_{\phi}u}{n+1}\langle\cdot,\cdot\rangle=0
\qquad \mathrm{on}~N.
\end{eqnarray}
Using the restriction of (\ref{4-15}) on the boundary, i.e.,
$\mathrm{Hess}z=0$ on $\partial N$, and $II=c\mathrm{I}_{n\times
n}$, one can get
\begin{eqnarray} \label{4-16}
h=cz
\end{eqnarray}
Combining (\ref{4-14}) and (\ref{4-16}) yields
\begin{eqnarray*}
\frac{\tau+\frac{c\gamma}{2}}{(K-1)c+2\gamma}=c,
\end{eqnarray*}
which, together with (\ref{4-13}), implies $\gamma=0$. This is a
contradiction. Hence, we have
\begin{eqnarray*}
\tau>\frac{\left[\sqrt{(K-1)^2c^2+2c(K-1)\gamma}+\gamma+(K-1)c\right]c}{2},
\end{eqnarray*}
which implies (\ref{TC4-2}). Besides, when $\beta=0$, (\ref{4-5}) is
equivalent with (\ref{TC4-2}). Summing up Cases 1 and 2, our
estimate (\ref{TC4-2}) can be achieved for the eigenvalue problem
(\ref{WEP}) with $\beta\geq0$.
\end{proof}

\begin{remark}
\rm{ When $\beta=0$, our estimate (\ref{TC4-1}) is exactly the main
estimate (1.6) in \cite{bs}, which means Theorem \ref{theorem3-1}
here covers \cite[Theorem 1.1]{bs} as a special case. When $\beta=0$
and $\phi=const.$ is a non-zero constant function, the Wentzell
eigenvalue problem (\ref{WEP}) of the weighted Laplacian degenerates
into the classical Steklov eigenvalue problem (\ref{SEP}), and
naturally $\tau_{1,\phi}=v_{1}$, i.e., the first non-zero Steklov
eigenvalue. In this setting, by Theorem \ref{theorem3-2}, one has
$v_{1}>\frac{c}{2}$, which is exactly Escobar's estimate in
\cite[Theorem 8]{e1}. That is to say, Theorem \ref{theorem3-2} here
covers Escobar's conclusion \cite[Theorem 8]{e1} as a special case.
}
\end{remark}

\section{An open problem}
\renewcommand{\thesection}{\arabic{section}}
\renewcommand{\theequation}{\thesection.\arabic{equation}}
\setcounter{equation}{0} \setcounter{maintheorem}{0}

If one checks Section \ref{PSEC} carefully, then naturally the
following interesting problem would be proposed.

\vspace{2mm}

\textbf{Open problem}. \emph{ For $n\geq4$, is the Escobar-type
Steklov eigenvalue inequality (\ref{e3}) also true without the
precondition (\ref{b1})?}

\vspace{5mm}

\section*{Acknowledgments}
\renewcommand{\thesection}{\arabic{section}}
\renewcommand{\theequation}{\thesection.\arabic{equation}}
\setcounter{equation}{0} \setcounter{maintheorem}{0}

This research was supported in part by the National Natural Science
Foundation of China (Grant Nos. 11401131 and 11801496), China
Scholarship Council, the Fok Ying-Tung Education Foundation (China),
the NSF of Hubei Provincial Department of Education (Grant No.
D20184301), and Hubei Key Laboratory of Applied Mathematics (Hubei
University). The eigenvalue comparisons here for the Steklov
eigenvalue problem have been already announced by Prof. Jing Mao in
a talk at School of Mathematical Sciences, Fudan University in May,
2017. The corresponding author, Prof. J. Mao, wants to thank the
Department of Mathematics, IST, University of Lisbon for its
hospitality during his visit from September 2018 to September 2019.


\begin{thebibliography}{99}



\bibitem{be} D. Bakry, M. \'{E}mery, \emph{Diffusion
hypercontractives}, S\'{e}m. Prob. XIX, Lect. Notes in Math. {\bf
1123} (1985) 177--206.


\bibitem{bbm} M. F. Betta, F Brock, A. Mercaldo, M. R. Posterar,
\emph{A weighted isoperimetric inequality and applications to
symmetrization}, J. Inequal. Appl. {\bf 4} (1999) 215--240.


\bibitem{bcp} M. Batista, M. P. Cavalcante, J. Pyo., \emph{Some isoperimetric
inequalities and eigenvalue estimates in weighted manifolds}, J.
Math. Anal. Appl. {\bf 419} (2014) 617--626.

\bibitem{bs} M. Batista, J. I. Santos, \emph{The first Stekloff eigenvalue in weighted Riemannian
manifolds}, available online at arXiv:1504.02630v1.

\bibitem{apc} A. P. Calder\'{o}n, \emph{On an inverse boundary value problem}, in ``Seminar
in Numerical Analysis and Its Applications to Continuum Physics'',
pp. 65--73, Soc. Brasileira de Matem\'{a}tica, Rio de Janeiro, 1980.


\bibitem{cheng1} S. Y. Cheng, \emph{Eigenvalue comparison theorems and its
geometric application}, Math. Z. {\bf 143} (1975) 289--297.

\bibitem{cheng2} S. Y. Cheng, \emph{Eigenfunctions and eigenvalues of the
Laplacian}, Am. Math. Soc. Proc. Symp. Pure Math. (Part II) {\bf 27}
(1975) 289--297.


\bibitem{DKL} M. Dambrine, D. Kateb, J. Lamboley, \emph{An extremal eigenvalue problem for the Wentzell-Laplace operator},
 Ann. I. H. Poincar\'e Non Linear Analysis {\bf 33}(2) (2014) 409-450.


\bibitem{dwx} F. Du, Q. L. Wang, C. Y. Xia, \emph{Estimates for eigenvalues of the Wentzell-Laplace operator},
 Journal of Geometry and Physics {\bf 129} (2018) 25-33.

 \bibitem{dmwx} F. Du, J. Mao, Q. L. Wang, C. Y. Xia, \emph{Isoperimetric bounds for eigenvalues of the Wentzell-Laplace, the Laplacian and a biharmonic Steklov
 problem}, submitted and available online at arXiv:1808.10578.


\bibitem{e1} J. F. Escobar, \emph{The geometry of the first non-zero Stekloff eigenvalue}, J. Funct. Anal.
{\bf 150} (1997) 544--556.

\bibitem{e2} J. F. Escobar, \emph{An isoperimetric inequality and the first Steklov eigenvalue}, J. Funct. Anal.
{\bf 165} (1999) 101--116.

\bibitem{e3} J. F. Escobar, \emph{A comparison theorem for the first non-zero Steklov
eigenvalue}, J. Funct. Anal. {\bf 178} (2000) 143--155.


\bibitem{hm} L. B. Hou, J. Mao, \emph{The principle of numerical calculations for the eigenvalue comparison on parameterized surfaces}, J. Math. Res. Appl. {\bf38} (2018) 58--62.

\bibitem{md} L. Ma, S. H. Du, \emph{Extension of Reilly formula with applications to eigenvalue
estimates for drifting Laplacians}, C. R. Math. Acad. Sci. Paris
{\bf 348} (2010) 1203--1206.


\bibitem{fmi} P. Freitas, J. Mao and I. Salavessa, \emph{Spherical symmetrization and the first eigenvalue of
geodesic disks on manifolds}, Calc. Var. Partial Differential
Equations {\bf 51} (2014) 701--724.

\bibitem{mg} M. Gromov, \emph{Isopermetric of Waists and concentration of
maps}, Geom. Funct. Anal. {\bf 13}(1) (2003) 178--215.


\bibitem{hr} Q. Huang, Q. H. Ruan, \emph{Application of some elliptic equations in Riemannian
manifolds},  J. Math. Anal. Appl. {\bf 409} (2014) 189--196.


 \bibitem{KK} N. N. Katz and K. Kondo, \emph{Generalized space forms},
Trans. Amer. Math. Soc. {\bf 354} (2002) 2279-2284.


\bibitem{jm1} J. Mao, \emph{Eigenvalue estimation and some results on finite topological
type}, Ph.D. thesis, IST-UTL, 2013.

\bibitem{jm2} J. Mao, \emph{Eigenvalue inequalities for the p-Laplacian on a Riemannian manifold and estimates for the heat kernel},
J. Math. Pures Appl. {\bf 101}(3) (2014) 372--393.

\bibitem{mdw} J. Mao, F. Du, C. X. Wu, \emph{Eigenvalue Problems on
Manifolds}, Science Press, Beijing, 2017.


\bibitem{k} B. O'Neill,
\emph{Semi-Riemannian Geometry with applications to relativity},
vol. 103 of Pure and Applied mathematics, Academic Press, San Diego,
1983.

\bibitem{p} P.\ Petersen, \emph{Riemannian Geometry}, Second Edition, vol.171 of Graduate
Texts in Mathematics, Springer, New York, 2006.


\bibitem{mws} M. W. Stekloff, \emph{Sur les probl\`{e}mes fondamentaux de la physique
math\'{e}matique}, Ann. Sci. \'{E}cole Norm. Sup. {\bf19} (1902)
455--490.


\bibitem{w1} R. Weinstock, \emph{Inequalities for a classical eigenvalue problem}, J. Rational
Mech. Anal. {\bf 3} (1954) 745--753.


 \bibitem{wx} C. Y. Xia, Q. L. Wang, \emph{Eigenvalues of the Wentzell-Laplace operator and of the fourth order Steklov problems}, Journal of Differential Equations
{\bf 264}(10) (2018) 6486-6506.


\end{thebibliography}
 \end{document}